\documentclass[11pt]{amsart}
\usepackage{amsmath,amssymb,amsthm,amscd,verbatim}
\usepackage{graphicx}
\usepackage{epsfig}
\usepackage{stmaryrd}
\numberwithin{equation}{section}

\newtheorem{thm}[equation]{Theorem}
\newtheorem{definition}[equation]{Definition}
\newtheorem{lem}[equation]{Lemma}

\newtheorem{claim}[equation]{Claim}
\newtheorem*{thm*}{Theorem}

\theoremstyle{remark}
\newtheorem{remark}[equation]{Remark}

 \newcommand{\A}{\mathcal{A}}
  \newcommand{\K}{\mathcal{K}}
   \newcommand{\G}{\mathcal{G}}
  
  \newcommand{\F}{\mathcal{F}}

  \newcommand{\RR}{\mathbb{R}}
    
        \newcommand{\RF}{\mathbb{R}[\mathcal{F}]}

  \newcommand{\by}{\mathbf{y}}
 \newcommand{\ind}{{\operatorname{ind}}}
  \newcommand{\inj}{{\operatorname{inj}}}
\newcommand{\id}{{\operatorname{id.}}}

\def\Ex {{\mathbb E}}

\begin{document}

\title {Undecidability of linear inequalities in graph homomorphism densities}

\author{Hamed Hatami}
\address{Department of Mathematics, Princeton University, Princeton, NJ}
\email{hhatami@math.princeton.edu \\ snorin@math.princeton.edu}

\author{Serguei Norine}
\thanks{Hamed Hatami was supported in part by NSERC. Serguei Norine was supported in part by NSF
  under Grant No. DMS-0701033}

\begin{abstract}
The purpose of this article is to show that even the most elementary  problems in asymptotic extremal graph theory can be highly non-trivial.
We study linear inequalities between graph homomorphism densities.
In the language of quantum graphs the validity of such an  inequality is equivalent to the positivity of a corresponding quantum graph. Similar to the setting of polynomials, a quantum graph that can be represented as a sum of squares of labeled quantum graphs is necessarily positive. Lov\'asz~(Problem 17 in \cite{openProblems}) asks whether the opposite is also true.  We answer this question and also a related question of Razborov in the negative by introducing explicit valid inequalities that do not satisfy the required conditions.
Our solution to these problems is based on a reduction from real multivariate polynomials and uses the fact that there are positive polynomials that cannot be expressed as sums of squares of polynomials.

It is known that the problem of determining whether a multivariate polynomial is positive is decidable. Hence it is very natural to ask ``Is the problem of determining the validity of a linear inequality between homomorphism densities decidable?''  We give a negative answer to this question which shows that such inequalities are inherently difficult in their full generality. Furthermore we deduce from this fact that the analogue of Artin's solution to Hilbert's seventeenth problem does not hold in the setting of quantum graphs.
\end{abstract}

\maketitle

\newline
{{\sc Keywords:} graph homomorphism density, quantum graph, decidability, Artin's theorem

\section{Introduction \label{sec:intro}}
Many fundamental theorems in extremal graph theory can be expressed as algebraic inequalities between subgraph densities. As it is explained below, for dense graphs, it is possible to replace subgraph densities with homomorphism densities.  An easy observation shows that one can convert any algebraic inequality between homomorphism densities to a linear inequality. Inspired by the work of Freedman, Lov\'asz and Schrijver~\cite{MR2257396}, in recent years a new line of research  in the direction of treating and understanding these inequalities in a unified way has emerged. Razborov~\cite{MR2371204} observed that a typical proof of an inequality in extremal graph theory between homomorphism densities of some fixed graphs involves only homomorphism densities of finitely many graphs. He states in~\cite{RazborovTuran} that
in his opinion the most interesting general open question about asymptotic extremal combinatorics is whether every true linear inequality between homomorphism densities can be proved using a finite amount of manipulation with homomorphism densities of finitely many graphs.
Although this question itself is not well-defined, a natural precise refinement is whether the problem of determining the validity of a linear inequality between homomorphism densities is decidable. We show that it is not. Our result in particular answers  various related questions by Razborov~\cite{MR2371204}, Lov\'asz~\cite{openProblems}, and Lov\'asz and Szegedy~\cite{Positivstellensatz}.

An interesting recent result in extremal graph theory, proved in several different forms~\cite{MR2257396, MR2371204, Positivstellensatz}, says that every linear inequality between homomorphism densities  follows from the positive semi-definiteness of a certain infinite matrix. As an immediate consequence, every algebraic inequality between the homomorphism densities follows from an infinite number of certain applications of the Cauchy-Schwarz inequality. This is consistent with the fact that many results in  extremal graph theory are proved by one or more tricky applications of the Cauchy-Schwarz inequality. Lov\'asz~\cite{openProblems} composed a collection of open problems in this area, and in Problem 17 he asks whether it is true or not that every algebraic inequality between homomorphism densities follows from a \emph{finite} number of applications of this inequality. It is possible to rephrase this question in the language of quantum graphs defined by Freedman, Lov\'asz and Schrijver~\cite{MR2257396}. The validity of a linear inequality between homomorphism densities corresponds to the positivity of a corresponding quantum graph. The square of a labeled quantum graph is trivially positive. In this language, Lov\'asz's question  translates to the following statement: ``Is it true that every positive quantum graph can be expressed as the sum of a finite number of squares of labeled quantum graphs?'' The question in this form is stated by Lov\'asz and Szegedy in~\cite{Positivstellensatz}. In Theorem~\ref{thm:Lovasz}, we show that the answer is negative.

In~\cite{MR2371204} Razborov introduced flag algebras which provide a powerful formal calculus that captures many standard arguments in extremal combinatorics. He presented several questions about the linear inequalities between homomorphism densities among which is a question about a calculus introduced by him  called the \emph{Cauchy-Schwarz calculus}. This calculus which allows trickier applications of the Cauchy-Schwarz inequality  can be used to prove the positivity of quantum graphs. He asks (\cite{MR2371204} Question 2) whether the Cauchy-Schwarz calculus is complete. We give a negative answer to this question by constructing positive quantum graphs whose positivity does not follow from this calculus.

A multivariate polynomial that takes only non-negative values over the reals is called  \emph{positive}.
Our solutions to Lov\'asz's seventeenth problem and Razborov's question about the Cauchy-Schwarz calculus are both based on reductions from real multivariate polynomials and they use the fact that there are positive polynomials that cannot be expressed as sums of squares of polynomials. Hence these answers are expected once one accepts the analogy to multivariate polynomials. However Artin~\cite{artin} solving Hilbert's seventeenth problem showed that every positive  polynomial  can be represented as a sum of squares of \emph{rational functions}.

In Theorem~\ref{thm:undecide} we prove that  determining the validity of a linear inequality between homomorphism densities  is undecidable. This reveals a major difference between the positivity of quantum graphs and the positivity of polynomials over reals as (for example by the celebrated work of Tarski~\cite{MR0028796}) it is known that the latter is decidable. Furthermore we deduce from this theorem that the analogue of Artin's solution to Hilbert's seventeenth problem does not hold in the setting of quantum graphs. This in particular answers Problem 21 of Lov\'asz's list of open problems~\cite{openProblems}.

Although our results show that not every algebraic inequality between homomorphism densities is a linear combination of  a finite number of semi-definiteness inequalities, the positive semi-definite characterization is still a powerful approach for proving such inequalities. Razborov in~\cite{RazborovTuran} illustrated the power of this method by applying it to prove various results (some new and some known) in extremal combinatorics. Razborov~\cite{RazborovTuran} and Lov\'asz and Szegedy~\cite{Positivstellensatz} observed that it is possible to use this method to verify every linear inequality between homomorphism densities within an  arbitrarily small error term. As this result suggests, the positive semi-definiteness method is extremely useful in proving bounds for problems in extremal combinatorics: Razborov~\cite{RazborovTuran} showed that a straightforward application of this method substantially improves the previously known bound for the Tur\'an's function of $K_4^3$, one of the most important problems in extremal combinatorics.

\section{Preliminaries \label{sec:prel}}
In this paper all graphs are simple and finite. For a graph $G$, let $V(G)$ and $E(G)$, respectively denote  the set of the vertices and the edges of $G$. The unique graph with no vertices is denoted by $\emptyset$. The \emph{density} of a graph $H$ in a graph $G$, denoted by $t_\inj(H;G)$, is the probability that a random \emph{embedding} of the vertices of $H$ in the vertices of $G$ maps every edge of $H$ to an edge of $G$. The \emph{homomorphism density} of $H$ in $G$, denoted by $t(H;G)$, is the probability that a random mapping (not necessarily injective) from the vertices of $H$ to the vertices of $G$ maps every edge of $H$ to an edge of $G$. We define $t_\inj(\emptyset;G) := t(\emptyset;G) :=1$, for every graph $G$.

Although $t(H;G)$ itself is an object of interest, extremal graph theory more often concerns $t_{\inj}(H;G)$. However,
the following simple lemma from~\cite{MR2274085}  shows that this two quantities are close up to an error term of $o(1)$, and hence are equivalent asymptotically.
\begin{lem}~\cite{MR2274085}
\label{lem:tAsym} For every two graphs $H$ and $G$,
$$|t(H;G)-t_\inj(H;G)| \le \frac{1}{|V(G)|} {|V(H)| \choose 2}=o_{|V(G)| \rightarrow \infty}(1).$$
\end{lem}

Many important results in extremal graph theory can be expressed as algebraic inequalities between homomorphism densities. For example Goodman's theorem~\cite{MR0107610}, which generalizes the classical Mantel-Tur\'an  Theorem, says that for every graph $G$, $t(K_3; G) \ge 2t(K_2; G)^2- t(K_2; G)$. Note that if $H_1 \dot\cup H_2$ denotes the disjoint union of two graphs $H_1$ and $H_2$, we have $t(H_1 \dot\cup H_2; G)=t(H_1; G) t(H_2; G)$. This observation allows us to convert any algebraic inequality between homomorphism densities to a linear inequality. For example one can restate Goodman's theorem as $t(K_3; G) - 2t(K_2 \dot\cup K_2; G) + t(K_2; G) \ge 0$.

A \emph{partially labeled graph} is a graph in which some of the vertices are labeled by distinct natural numbers (there may be any number of unlabeled vertices). Let ${\mathcal F}$ denote the set of all partially labeled graphs up to label-preserving isomorphism. A partially labeled graph in which all vertices are labeled is called a \emph{fully labeled} graph. The product of two partially labeled graphs $H_1$ and $H_2$, denoted by $H_1 \cdot H_2$, is defined by taking their disjoint union, and then identifying vertices with the same label (if multiple edges arise, only one copy is kept). Clearly this multiplication is associative and commutative, and thus turns ${\mathcal F}$ into a commutative semi-group. For every finite subset $L$ of natural numbers, let ${\mathcal F}_L$ denote the sub-semi-group of all partially labeled graphs whose set of labels is exactly $L$. Note that  ${\mathcal F}_\emptyset$ is the set of all finite graphs with no labels, and the product of two graphs in ${\mathcal F}_\emptyset$ is their disjoint union.

We extend the definition of homomorphism density to partially labeled graphs in the following way. Consider a finite set $L \subset \mathbb{N}$, a partially labeled graph $H \in {\mathcal F}_L$, a graph $G$, and a map $\phi:L \to V(G)$. Then $t(H;G, \phi)$ is defined to be the probability that a random map from $V(H)$ to $V(G)$ is a homomorphism conditioned on the event that the labeled vertices are mapped according to $\phi$. Note that  for every two partially labeled graphs $H_1,H_2 \in {\mathcal F}_L$, a graph $G$, and a map $\phi:L \to V(G)$,
\begin{equation}
\label{eq:product}  t(H_1 \cdot H_2 ;G, \phi) = t(H_1;G, \phi) t(H_2;G, \phi).
\end{equation}

A \emph{graph parameter} is a function that maps every graph to a real number. For example, given any graph $G$, the function $f:H \mapsto t(H; G)$ is a graph parameter. Freedman, Lov\'asz, and Schrijver~\cite{MR2257396} proved that
it is possible to  characterize the graph parameters that are defined in a similar fashion using  some positive semi-definiteness and rank conditions. Since then, various similar characterizations have been found. In particular, the following statement is proved in~\cite{Positivstellensatz}: Let $f$ be a graph parameter. There exists a sequence of graphs $\{G_n\}_{n \in \mathbb{N}}$ such that  $\lim_{n \rightarrow \infty} t(H; G_n) = f(H)$, for every $H$, if and only if the following conditions are satisfied:
\begin{itemize}
\item [{\bf (i)}] $f(\emptyset)=1$, and $f(H_1)=f(H_2)$, if $H_1$ is obtained from $H_2$ by adding isolated vertices.
\item [{\bf (ii)}] The ${\mathcal F} \times {\mathcal F}$ matrix whose entry in the intersection of the row corresponding to $H_1$ and the column corresponding to $H_2$ is $f(H_1 \cdot H_2)$ is positive semi-definite. (Here the labels of $H_1 \cdot H_2$ are removed.)
\end{itemize}

Consider real numbers $\alpha_1,\ldots,\alpha_k$ and graphs $H_1, \ldots, H_k$. We wish to investigate whether an  inequality of the form
\begin{equation}
\label{eq:linearIneq}
\alpha_1 t(H_1;G)+\ldots+\alpha_k t(H_k;G) \ge 0,
\end{equation}
holds for all graphs $G$. The positive semi-definiteness characterization shows that it suffices to verify the validity of $\alpha_1 f(H_1)+\ldots+\alpha_k f(H_k) \ge 0$, for all graph parameters $f$ satisfying Conditions~(i) and~(ii).  Note that if the ${\mathcal F} \times {\mathcal F}$ matrix in Condition (ii) was finite, then there would exist an algorithm for solving this problem using semi-definite programming  (see~\cite{MR1952986} for a survey on this topic). However since this matrix is of infinite dimensions, in practice one can only restrict to a finite sub-matrix of it and hope that $\alpha_1 f(H_1)+\ldots+\alpha_k f(H_k) \ge 0$ is still valid if the weaker condition that this sub-matrix is positive semi-definite is required. The examples discussed in Section~\ref{sec:intro} show that this method is indeed very powerful and in many cases one succeeds in finding a proof for such inequalities. Lov\'asz in~\cite{openProblems}  asks (Problem 17) whether the validity of (\ref{eq:linearIneq}) always follows from the positive semi-definiteness of a finite sub-matrix of the ${\mathcal F} \times {\mathcal F}$ matrix in Condition~(ii). As we will see in Section~\ref{sec:Lovasz17}, it is possible to reformulate this question in the language of quantum graphs.

A \emph{labeled quantum graph} is an element of the algebra $\mathbb{R}[{\mathcal F}]$, i.e. it is a formal linear combination of partially labeled graphs, and if $f = \sum_{i=1}^n \alpha_i H_i \in \mathbb{R}[{\mathcal F}]$ and $g = \sum_{i=1}^m \beta_i F_i \in \mathbb{R}[{\mathcal F}]$, then $f \cdot g =  \sum_{i=1}^n \sum_{j=1}^m \alpha_i \beta_j  H_i \cdot F_j$. For every finite set of positive integers $L$, $\mathbb{R}[{\mathcal F}_L]$ is a sub-algebra of $\mathbb{R}[{\mathcal F}]$. The elements of $\mathbb{R}[{\mathcal F}_\emptyset]$ are called \emph{quantum graphs}.

Consider a labeled quantum graph $f = \sum_{i=1}^k \alpha_i H_i \in \mathcal{F}$, and for every $i \in [k]$, let $L_i$ be the set of all labels appearing on  $H_i$. For a graph $G$ and a map $\phi:\cup_{i=1}^k L_i \to V(G)$, define $t(f;G, \phi) := \sum_{i=1}^k \alpha_i t(H_i;G, \phi|_{L_i})$.
Let $\mathcal{K}$ be the ideal of $\mathbb{R}[{\mathcal F}]$ generated by elements of the form $F-\emptyset$, where $F$ is a possibly labeled $1$-vertex graph. Note that $\K$ is the linear subspace of $\mathbb{R}[{\mathcal F}]$ spanned by elements of the form $F-H$ where $F$ is obtained from $H$ by adding a possibly labeled isolated vertex.  Hence $t(f;G, \phi) = 0$, if $f \in \K$, and  the function $t(\cdot;G, \phi)$ is a well-defined map from the quotient algebra $\A := \mathbb{R}[\mathcal{F}]/ \K$ to $\RR$.

Consider a finite set $L \subset \mathbb{N}$, a graph $G$, and a map $\phi:L \rightarrow V(G)$.
It follows from (\ref{eq:product}) that $f \mapsto t(f;G, \phi)$ defines a homomorphism from  $\mathbb{R}[{\mathcal F}_L]$ to $\RR$, and hence it is also   a well-defined homomorphism from $\A_{L} := \mathbb{R}[\mathcal{F}_L]/(\K \cap \mathbb{R}[{\mathcal F}_L])$ to $\RR$.

For every finite set of positive integers $L$, let the linear map $\llbracket  \cdot \rrbracket_L: \mathbb{R}[{\mathcal F}] \rightarrow \mathbb{R}[{\mathcal F}]$ be defined by un-labeling the vertices whose labels are \emph{not in}  $L$. Note that this map is not an algebra homomorphism, as it does not respect the product. However $\llbracket  \cdot \rrbracket_L$ maps $\K$ to $\K$, so we can consider $\llbracket  \cdot \rrbracket_L$ as a linear map from $\A$ to itself. We abbreviate $\llbracket  \cdot \rrbracket_\emptyset$ to $ \llbracket  \cdot \rrbracket$.

\subsection{Lov\'asz's seventeenth problem\label{sec:Lovasz17}}
We say that a labeled quantum graph $f \in \mathbb{R}[{\mathcal F}_L]$ is \emph{positive} and write $f \geq 0$, if for every graph $G$ and every $\phi: L \to V(G)$ we have $t(f;G, \phi) \geq 0$. By the discussion above, we can extend the definition of positivity to $\A_L$ and further to the whole of $\A$.

Note that $g^2 \ge 0$, for every $g \in \mathbb{R}[{\mathcal F}_L]$. Furthermore for every subset $S \subseteq L$, the map $\llbracket \cdot \rrbracket_S$ preserves positivity. It follows that  $\llbracket g^2 \rrbracket \geq 0$, for every $g \in \mathbb{R}[{\mathcal F}]$. Hence one possible approach to prove an inequality of the form (\ref{eq:linearIneq}) is to express $f = \sum_{i=1}^k \alpha_i H_i$ as a sum of squares of labeled quantum graphs, i.e. to find  labeled quantum graphs $g_1,\ldots,g_m$ such that  $\sum_{i=1}^k \alpha_i H_i = \sum_{i=1}^m \llbracket  g_i^2 \rrbracket$.  Lov\'asz's seventeenth problem asks whether every positive quantum graph can be expressed in this form. In Section~\ref{sec:Lovasz} we prove the following theorem which answers this question in the negative. We say that $x \in \A_{\emptyset}$ is \emph{expressible as a sum of squares} if there exist $g_1,\ldots, g_m \in \A$ such that $x = \llbracket \sum_{i=1}^m g_i^2 \rrbracket$.

\begin{thm}
\label{thm:Lovasz}
There exists a positive quantum graph which cannot be expressed as a sum of squares.
\end{thm}

Whitney~\cite[Theorem 5a]{MR1503085} has shown that the functions $t_\inj(H;\cdot)$, $H$ connected, $V(H)>1$, are algebraically independent. Equivalently, the functions $t(H;\cdot)$, $H$ connected, $V(H)>1$, are algebraically independent, as one can straightforwardly verify that a non-trivial algebraic relation between the functions $t_\inj(H;\cdot)$ would imply a non-trivial algebraic relation between functions $t(H;\cdot)$. (In fact, it is shown in~\cite{MR538044} that the functions $t(H;\cdot)$ are independent in even stronger sense.) It follows that
\begin{equation}
\label{eq:whitney}
\K \cap \mathbb{R}[{\mathcal F}_{\emptyset}]= \{f \in \mathbb{R}[\mathcal{F}_\emptyset]: \mbox{$t(f;G) = 0$  for every graph $G$}\}.
\end{equation}
Therefore Theorem~\ref{thm:Lovasz} is equivalent to the existence of a positive quantum graph $x$ such that  for every collection  $g_1,\ldots,g_m \in \RF$, there exists a graph $G$ so that $$t(x;G) \neq t\left(\sum_{i=1}^m \llbracket  g_i^2  \rrbracket;G \right).$$

\subsection{Artin's theorem}
Note that Theorem~\ref{thm:Lovasz} reminisces Hilbert's classical theorem that there exists positive multivariate real polynomials that cannot be expressed as sums of squares of polynomials. Hilbert in the seventeenth problem of his celebrated list of open problems asked ``Given a multivariate polynomial that takes only non-negative values over the reals, can it be represented as a sum of squares of \emph{rational functions}?'' In 1927 Emil Artin~\cite{artin} answered this question in the affirmative. Note that Artin's theorem is equivalent to the fact that for every multivariate polynomial $p$ that takes only non-negative values over the reals, there exists polynomials $q \neq 0$ and $r$, each expressible as a sum of squares of polynomials, such that $q p =r$.  Our proof of Theorem~\ref{thm:Lovasz} relies on the above mentioned theorem of Hilbert. Hence it is very natural to wonder whether the analogue of Artin's theorem holds for quantum graphs. Indeed Lemma~\ref{lem:preArtin} below (proved in Section~\ref{sec:Artin}) shows that the validity of such a statement would imply a simple finitary characterization of positive quantum graphs.

\begin{lem}
\label{lem:preArtin}
If $x \in \A_{\emptyset}$ satisfies $gx=h$ for some positive $g,h \in \A_{\emptyset}$ with $g \neq 0$, then $x$ is positive.
\end{lem}

In Section~\ref{sec:Lovasz} we prove the following theorem which shows that the analogue of Artin's theorem for quantum graphs does not hold.

\begin{thm}
\label{thm:Artin}
There exists positive $x \in \A_{\emptyset}$ such that there are no $g,h \in \A_{\emptyset}$, each expressible as a sum of squares, with $g \neq 0$, so that $g x = h$.
\end{thm}

Note that Theorem~\ref{thm:Artin} implies Theorem~\ref{thm:Lovasz}. However since our proof of Theorem~\ref{thm:Artin} is based on the undecidability result proved in Theorem~\ref{thm:undecide} below and hence does not provide any explicit examples, we give a separate constructive proof of Theorem~\ref{thm:Lovasz}.

\begin{remark}
\label{rem:prob21}
In~\cite{openProblems} Problem 21, Lov\'asz asked ``Is it true that for every positive quantum graph $x$, there exist quantum graphs $g$ and $h$, each expressible as a sum of squares of labeled quantum graphs, so that $x + gx = h$?'' It follows from Theorem~\ref{thm:Artin} that the answer to this question is also negative.
\end{remark}
\vskip 10pt

\subsection{Razborov's Cauchy-Schwarz Calculus}

Consider $f_1,f_2 \in \A_L$, and a subset $T \subseteq L$. Razborov  (\cite{MR2371204} Theorem~3.14) has shown that
\begin{equation}
\label{eq:RazCauchy}
\llbracket f_1^2\rrbracket_T \llbracket f_2^2\rrbracket_T \ge \llbracket f_1f_2\rrbracket_T^2.
\end{equation}
The \emph{Cauchy-Schwarz calculus} is defined in~\cite{MR2371204} in the language of flag algebras, but can be reformulated as follows.

\begin{definition}
\label{def:CS}
The Cauchy-Schwarz calculus  operates with statements of the form $f \geq 0$ with $f \in \A$ and has axioms
\begin{itemize}
\item [{\bf A1: }] $f^2 \geq 0$ for every $f$;
\item [{\bf A2: }] $\llbracket f_1^2\rrbracket_T \llbracket f_2^2\rrbracket_T - \llbracket f_1f_2\rrbracket_T^2 \geq 0$ for $f_1,f_2$ and $T$ as in~(\ref{eq:RazCauchy}).
\end{itemize}
The inference rules of the Cauchy-Schwarz calculus are
\begin{itemize}
\item[{\bf R1: }] $$ \frac{f \geq 0 \qquad g\geq 0}{\alpha f + \beta g \geq 0}\; (\alpha,\beta \geq 0),  $$
\item[{\bf R2: }] $$ \frac{f \geq 0 \qquad g\geq 0}{fg \geq 0}, $$
\item[{\bf R3: }] $$\frac{f \geq 0}{\llbracket f \rrbracket_T \geq 0}\; (T \subseteq \mathbb{N}),$$
\end{itemize}
for $f,g \in \A$.
We say that $f \in \A$ is \emph{CS-positive} if the statement $f \geq 0$ is provable in the Cauchy-Schwarz calculus.
\end{definition}

The original definition of the Cauchy-Schwarz calculus in~\cite{MR2371204} appears to differ from the one presented here. However as we shall discuss in Appendix~\ref{appendix}, the two definitions are equivalent in that a statement can be proven in the original calculus if and only if it can be proven using the one stated in Definition~\ref{def:CS}.

Answering a question of Razborov, in Section~\ref{sec:Lovasz} we show that the Cauchy-Schwarz calculus is not complete by proving the following theorem.

\begin{thm}
\label{thm:Razborov}
There exists a positive $f \in \A_{\emptyset}$ which is not CS-positive.
\end{thm}

\subsection{A tenth problem: Undecidability}
The proofs of both Theorems~\ref{thm:Lovasz} and~\ref{thm:Razborov} rely on reductions from the setting of multivariate polynomials. It follows from Artin's solution to Hilbert's 17th problem~(see Theorem 2.1.12 in \cite{MR1829790}) that every positive $p \in \mathbb{Q}[x_1,\ldots,x_n]$ can be expressed as a sum of the form $\sum_{i=1}^m a_i q_i^2$, where $a_i \in \mathbb{Q}_{+}$ and $q_i \in \mathbb{Q}(x_1,\ldots,x_n)$ for every $1 \le i \le m$.  Note that this in particular shows that the problem of determining whether a multivariate polynomial with rational coefficients is positive or not is decidable. Indeed given such a polynomial, one can search for expressing it as a sum of the form $\sum_{i=1}^m a_i q_i^2$ (there are countably many of those sums), and in parallel check its non-negativity on rational points. Eventually either one will find a way to express the polynomial as a sum of squares of rational functions or a rational point on which the polynomial takes a negative value will be found. The decidability of positivity of a polynomial with rational coefficients follows also from the well-known work of Tarski~\cite{MR0028796}. In the following theorem we show that the problem of determining the positivity of a quantum graph is undecidable.

\begin{thm}
\label{thm:undecide}
The following problem is undecidable.
\begin{itemize}
 \item {\sc instance:} A positive integer $k$, finite graphs $H_1,\ldots,H_k$, and integers $a_1,\ldots,a_k$.
 \item {\sc question:} Does the inequality $a_1t(H_1;G)+\ldots+a_k t(H_k;G) \ge 0$ hold for every graph $G$?
\end{itemize}
\end{thm}

For two graphs $H$ and $G$, let $\hom(H;G)$ denote the number of homomorphisms from $H$ to $G$. Note that $\hom(H;G) = |V(G)|^{|V(H)|} t(H;G)$. In~\cite{Rossman} it is observed that the undecidability of the following problem follows from the undecidability of a similar problem in database theory~\cite{211419}. Given a positive integer $k$, finite graphs $H_1,\ldots,H_k$, and integers $a_1,\ldots,a_k$, it is undecidable whether the following inequality holds for all graphs $G$: $$a_1\hom(H_1;G)+\ldots+a_k \hom(H_k;G) \ge 0.$$
Note that  Theorem~\ref{thm:undecide} in particular implies this result. Indeed since $t(H;G) = \frac{\hom(H;G)}{\hom(K_1;G)^{|V(H)|}}$ for all graphs $H$ and $G$, it is possible to  express any linear inequality in homomorphism densities as an algebraic inequality in homomorphism numbers which in turn can be converted into a linear inequality.

\section{Some auxiliary facts}\label{sec:auxiliary}

We start by defining some new notations and proving auxiliary facts.
The \emph{induced homomorphism density} of $H$ in $G$, denoted by $t_\ind(H;G)$, is the probability that a random map  from the vertices of $H$ to the vertices of $G$ preserves both adjacency and non-adjacency. The two functions $t(H,\cdot)$ and $t_\ind(H,\cdot)$ are related by
\begin{equation}
\label{eq:partialOrder}
t(H;\cdot) = \sum_{\substack{F \supseteq H \\ V(F)=V(H)}}  t_\ind(F;\cdot),
\end{equation}
and a M\"obius inversion formula
\begin{equation}
\label{eq:Mobius}
t_\ind(H;\cdot) = \sum_{\substack{F \supseteq H \\ V(F)=V(H)}}  (-1)^{|E(F) \backslash E(H)| } t(F;\cdot).
\end{equation}
For a partially labeled graph $H$, define the quantum graph
$$\ind(H) := \sum_{\substack{F \supseteq H \\ V(F)=V(H)}} (-1)^{|E(F) \backslash E(H)| }  F.$$
The labeled quantum graphs $\ind(H)$ enjoy certain orthogonality properties. Indeed if the restriction of  two partially labeled graphs $H_1, H_2 \in {\mathcal F}_L$ to the labeled vertices are different, then we have
\begin{equation}
\label{eq:orthogonality}
\ind(H_1) \cdot \ind(H_2) = 0.
\end{equation}

In Section~\ref{sec:prel} we defined the homomorphism density of a graph $H$ in a graph $G$. Sometimes we shall work in a slightly more general setting that allows $G$ to have a non-uniform distribution on its set of vertices. More precisely, let $H$ be a partially labeled graph with the set of labels $L \subset \mathbb{N}$. Let $G$ be another graph, $\by$ be a probability measure on the vertices of $G$ and $\phi: L \to V(G)$ be a map. Define the random mapping $h$ from the vertices of $H$ to the vertices of $G$ by mapping every unlabeled vertex of $H$ to a vertex of $G$ independently and according to the probability measure $\by$, and mapping the labeled vertices according to $\phi$. Then $t(H;G, \by,\phi)$ is the probability that $h$ defines a homomorphism from $H$ to $G$. Also $t_\inj(H;G, \by,\phi)$ and $t_\ind(H;G, \by,\phi)$ are defined similarly.

\begin{remark}\label{rem:rem1}
Consider a graph $G$ and a probability measure $\by$ on $V(G)$. For every positive integer $n$, construct the graph $G_n$ in the following way. For every vertex $v$ in $V(G)$, put $\lfloor \by(v) n \rfloor$ ``copies'' of $v$  in $G_n$. There is an edge between two vertices in $G_n$ if and only if they are copies of adjacent vertices in $G$. It is easy to see that for every graph $H$, we have $\lim_{n \rightarrow \infty} t(H;G_n) = t(H;G,\by)$. Hence a quantum graph $f$ is positive, if and only if $t(f;G,\by) \ge 0$ for every graph $G$ and every probability distribution $\by$ on $V(G)$. Similarly it is easy to see that a labeled quantum graph $f$ is positive, if and only if $t(f;G, \by,\phi)$ is always non-negative.
\end{remark}

Consider a graph $H$ with $V(H)=[k]$. For every  $j \in [k]$, let $H_j \in {\mathcal F}_{[k]}$ be obtained from $H$ by adding an unlabeled clone $v$ of the vertex $j$ to this graph, i.e. $v$ is adjacent exactly to the neighbors of the vertex $j$. Let $H_j'$ be obtained from $H_j$ by adding the edge $\{v,j\}$.  Then  $\varphi_H:\RR[x_1,\ldots,x_k] \rightarrow  \RR[\mathcal{F}_{[k]}]$ is the unique algebra homomorphism defined by $\varphi_H(x_j) := \ind(H_j) +  \ind(H_j')$, for every $j \in [k]$.

Consider a graph $G$ and a probability measure $\by$ on $V(G)$. Let $S$ denote the set of all maps $h:V(H) \to V(G)$ that preserve both adjacency and non-adjacency. Consider a map $\phi:V(H) \rightarrow V(G)$. If $\phi \not\in S$, then $t(\varphi_H(x_j);G, \by,\phi)=0$, for every $j \in [k]$. Hence in this case  $t(\varphi_H(p);G, \by,\phi)=0$, for every polynomial $p \in \RR[x_1,\ldots,x_k]$.

Next consider a map $\phi \in S$. Then for every $j \in [k]$, $t(\varphi_H(x_j);G, \by,\phi)$ is the probability that a random (according to $\by$) extension $\psi$ of $\phi$ in $H_j$ preserves  the adjacencies and non-adjacencies except maybe between $j$ and its clone. For every $j \in [k]$, let $\alpha_j(\phi)$ be the probability that the map obtained from $\phi$ by replacing $\phi(j)$ by a random vertex in $G$ chosen according to the probability measure $\by$ belongs to $S$. Since the unlabeled vertex of $H_j$ is a clone of the vertex $j$, we have $t(\varphi_H(x_j);G, \by,\phi) = \alpha_j(\phi)$.

We conclude that  for every map $\phi:[k] \rightarrow V(G)$, and every polynomial $p \in \RR[x_1,\ldots,x_k]$,
\begin{equation}
\label{eq:tHpG}
t(\varphi_H(p);G, \by,\phi) = \left\{ \begin{array}{lcl} p(\alpha_1(\phi),\ldots,\alpha_k(\phi)) & \qquad & \phi \in S\\
0 &&  \phi \not\in S \end{array} \right.
\end{equation}

 The next lemma follows immediately from (\ref{eq:tHpG}).

\begin{lem}\label{lem:L1}
Given a graph $H$ with $V(H)=[k]$  and a positive polynomial $p$ in $k$ variables, the labeled quantum graph $\varphi_H(p)$ is positive.
\end{lem}

Let $P_k$ denote the set of all positive \emph{homogenous} polynomials in $k$ variables. Let $\Sigma_k$ denote the set of those homogenous polynomials in $k$ variables that  can be expressed as sums of squares of polynomials with real coefficients. Clearly $\Sigma_k \subseteq P_k$. Let $\Delta_k = P_k \backslash  \Sigma_k$. Hilbert~\cite{MR1510539} showed that for $k \ge 3$, $\Delta_k$ is not empty.
An $x_0 \in \RR^k$ is called~\cite{MR1747589} a \emph{bad point} for polynomial $p \in P_k$, if $x_0$ is a root of every polynomial $q$ such that $q^2p \in \Sigma_k$.

\begin{lem}\label{lem:L2}
For every $k \geq 4$, there exists an even $p \in \Delta_k$  such that $x^2_1x^2_2 \ldots x_k^2$ divides $p$ and $(1,0,\ldots,0)$ is a bad point for $p$.
\end{lem}
\begin{proof}
Consider the polynomial $p:= x^2_1x^2_2 \ldots x_k^2 S(x_2,x_3,x_4)$, where $S(x,y,z)$ is an even homogeneous polynomial in $\Delta_3$, e.g. $S(x,y,z) := x^4 y^2 + y^4 z^2 + z^4 x^2  - 3 x^2 y^2 z^2$. A short argument that this particular polynomial belongs to $\Delta_3$ is provided for example in~\cite[p. 519]{MR2104929}.

Note that $p$ is  trivially positive. Suppose for the contradiction that $q^2 p = \sum_{i=1}^m q_i^2$ for polynomials $q_i$ while $q(1,0,\ldots,0)=c \neq 0$ for a polynomial $q$ of degree $d$. Note that $q(1,0,\ldots,0)=c$ implies that the monomial $x_1^d$ appears with coefficient $c$ in $q$. Hence the component of $x_1^{2d+2}$ in $q^2 p$ is $c^2 \left(x^2_2 \ldots x_k^2\right) S(x_2,x_3,x_4)$. Denoting by $\overline{q}_i$ the component of $x_1^{d+1}$ in $q_i$, we must have $c^2 \left(x^2_2 \ldots x_k^2\right) S(x_2,x_3,x_4) = \sum_{i=1}^m \overline{q}_i^2$. Note that as the left side of the equality is divisible by $x_j$ for each $2 \leq j \leq k$ we can readily see (by setting $x_j=0$) that every $\overline{q}_i$ is divisible by every $x_j $ and consequently by $x_2\ldots x_k$. It now follows that  $S(x_2,x_3,x_4) \in \Sigma_3$, which is a contradiction.
\end{proof}

\section{Proofs of Theorems~\ref{thm:Lovasz}~and~\ref{thm:Razborov}}\label{sec:Lovasz}

For finite $L \subset \mathbb{N}$, $f \in \mathbb{R}[\F_L]$ and a graph $G$ on $n$ vertices we say that $f$ is \emph{$G$-sos} if for every $\phi: L \to V(G)$ there exists polynomials $p_1,\ldots,p_m \in \mathbb{R}[y_1,\ldots,y_n]$ with \emph{non-negative coefficients}, and $q_1,\ldots,q_m \in \mathbb{R}[y_1,\ldots,y_n]$ such that for all probability distributions $\by$,
$$t(f;G, \by,\phi) = \sum_{i=1}^m p_i(\by) q_i(\by)^2.$$

The definition of $G$-sos extends to $\A$.
The following lemma which is a key step in the proofs of Theorems~\ref{thm:Lovasz}~and~\ref{thm:Razborov} relates the squares of labeled quantum graphs to the squares of polynomials.

\begin{lem}
\label{lem:key}
Let $G$ be a graph. Then
\begin{itemize}
 \item[(i)] $f^2$ is $G$-sos for every $f \in \A$;
 \item[(ii)] $\llbracket f^2 \rrbracket_T \llbracket g^2 \rrbracket_T - \llbracket fg \rrbracket_T^2$ is $G$-sos for all finite $T \subseteq \mathbb{N}$ and $f,g \in \A$;
 \item[(iii)] if $f$ and $g$ are $G$-sos then $\alpha f + \beta g$ is $G$-sos for all $\alpha,\beta \in \mathbb{R}_+$;
 \item[(iv)] if $f$ and $g$ are $G$-sos then $fg$ is $G$-sos;
 \item[(v)] if $f$ is $G$-sos, then $\llbracket f \rrbracket_T$ is $G$-sos for all  finite $T \subseteq \mathbb{N}$.
\end{itemize}
\end{lem}
\begin{proof}
To verify assertion (i), take $m=1$, $p_1=1$, and
$q_1=t(f;G,\by,\phi)$ in the definition of a $G$-sos element.
 For (ii), without loss of generality we assume that $f,g \in \A_L$ for some finite $L \supseteq T$. Denoting $A:=t(\llbracket f^2 \rrbracket_T \llbracket g^2 \rrbracket_T - \llbracket fg \rrbracket_T^2;G, \by,\phi)$, we have
\begin{eqnarray*}
A &=& \left( \sum_{\psi} \left(\prod_{v \in L \backslash T} y_{\psi(v)}\right) t^2(f ; G,\by,\psi)  \right)  \left( \sum_{\psi} \left(\prod_{v \in  L \backslash T} y_{\psi(v)}\right) t^2(g ; G,\by,\psi)  \right)  \\
&& -\left( \sum_{\psi} \left( \prod_{v \in  L \backslash T} y_{\psi(v)}\right) t(f; G,\by, \psi)t(g; G,\by, \psi)\right)^2  \\
&=& \frac{1}{2}\sum_{\psi_1} \sum_{\psi_2} \left(\prod_{v \in  L \backslash T} y_{\psi_1(v)}\right)\left(\prod_{v \in  L \backslash T} y_{\psi_2(v)}\right)\\
&& \times \left( t(f; G,\by, \psi_1)t(g ; G,\by, \psi_2) -  t(f ; G,\by, \psi_2)t(g ; G,\by, \psi_1) \right)^2,
\end{eqnarray*}
where the summations over $\psi,\psi_1$ and $\psi_2$ are over all maps from $L$  to $V(G)$, which coincide with $\phi$ on $T$.
Assertions (iii) and (iv) are trivial. Finally, for (v) if $f \in \mathbb{R}[\mathcal{F}_L]$, then we have
$$
\nonumber t(\llbracket f \rrbracket_T;G, \by,\phi) = \sum_{\psi}\left(\prod_{v \in  L \backslash T} y_{\psi(v)}\right) t(f, \psi; G,\by),
$$
where the summation is over all maps $\psi$ from $L$  to $V(G)$ which coincide with $\phi$ on $T$.
\end{proof}

Consider a graph $H$. A set $W \subseteq V(H)$ is called \emph{homogenous} in $H$, if for every two distinct vertices $u,v \in W$, $N(u) \backslash W \neq N(v) \backslash W$, where $N(u)$ and $N(v)$ respectively denote the set of the neighbors of $u$ and $v$. We call a graph $H$ \emph{stringent}, if it does not contain any homogenous subsets $W$ with $1<|W| \le |V(H)|-1$, and furthermore does not allow any non-trivial automorphisms. Note that if $H$ is stringent, then in particular the identity map is the only map from $H$ to itself that preserves both adjacency and non-adjacency.

Stringent graphs serve as the foundation for all our constructions, and so we will need the following simple lemma.

\begin{figure}
\centering
\includegraphics[scale=0.75]{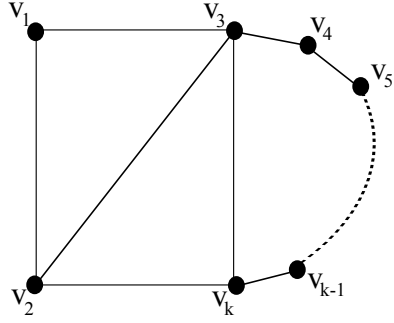} 
\caption{A stringent graph on $k \geq 6$ vertices.} \label{f:asymmetric}
\end{figure}
\begin{lem}\label{lem:stringent}
For every $k \geq 6$, there exists a stringent graph on $k$ vertices.
\end{lem}

\begin{proof} Let the graph $H$ on $k \geq 6$ vertices be obtained from a triangle with vertices $v_1v_2v_3$ by adding a path $v_3v_4\ldots v_k$ and finally by joining $v_k$ to $v_2$ and $v_3$. See Figure~\ref{f:asymmetric}.

Suppose first that $W \subseteq V(H)$ with $1<|W|\le |V(H)|-1$ is homogenous. Then $|W| \neq |V(H)|-1$, as no vertex of $H$ has $|V(H)|-1$ neighbors. The graph $H$ is $2$-connected, and therefore $W$ has at least two neighbors outside of $W$. It follows that $W$ must lie in the intersection of the neighborhoods of two vertices in $V(H) \backslash W$. This allows for a single possibility, namely $W=\{v_2,v_3\}$, which is clearly non-homogenous. Similarly, one can routinely verify that $H$ allows no non-trivial automorphisms, and thus is stringent.
\end{proof}

\begin{proof}[Proof of Theorem~\ref{thm:Lovasz}]

Let $H$ be a stringent graph on $k \ge 4$ vertices. Let $V(H) = [k]$. By (\ref{eq:tHpG}) for every homogenous polynomial $p \in \RR[x_1,\ldots,x_k]$, and every map $\phi: [k] \to V(H)$,
\begin{equation}
\label{eq:tHpHby}
t(\varphi_H(p); H,\by,\phi) = \left\{ \begin{array}{lcl} p(y_1,\ldots,y_k) & \qquad & \phi = \id \\
0 &&  \phi \neq \id \end{array} \right.
\end{equation}
Since the probability that a random map $\phi$ picked according to the probability measure $\by$ is equal to the identity map is $y_1y_2\ldots y_k$, it follows from (\ref{eq:tHpHby}) that
\begin{equation}
\label{eq:tHpNew}
t(\llbracket \varphi_H(p) \rrbracket; H,\by) =  y_1y_2\ldots y_k p(y_1,\ldots,y_k).
\end{equation}
Let $p$ be a homogeneous polynomial in $k$ variables such that the
even homogeneous polynomial $x^2_1x^2_2 \ldots x_k^2p(x_1^2,\ldots,x_k^2)$ satisfies the
assertion of Lemma~\ref{lem:L2}, e.g.
$$p(x_1^2,\ldots,x_k^2) = x_2^4 x_3^2 + x_3^4 x_4^2 + x_4^4 x_2^2  - 3 x_2^2 x_3^2 x_4^2.$$
By Lemma~\ref{lem:L1}, the quantum graph $x:=\llbracket \varphi_H(p) \rrbracket$ is positive. We claim that $x$ satisfies the assertion of  Theorem~\ref{thm:Lovasz}. Assume to the contrary that there exist labeled quantum graphs $g_1,\ldots,g_m$, such that for every graph $G$,
$$t(x;G) = \sum_{i=1}^m t(\llbracket g_i^2 \rrbracket; G).$$
Then by Remark~\ref{rem:rem1}, for every graph $G$ and every probability distribution $\by$ on the vertices of $G$,
$$t(x;G,\by) = \sum_{i=1}^m t(\llbracket g_i^2 \rrbracket; G,\by).$$
By Lemma~\ref{lem:key} there exist polynomials $p_1,\ldots,p_m \in \mathbb{R}[y_1,\ldots,y_k]$ with non-negative coefficients, and $q_1,\ldots,q_m \in \mathbb{R}[y_1,\ldots,y_k]$ so that
\begin{equation}
\label{eq:thmLovasz-e1}
t(x; H,\by) = \sum_{i=1}^m p_i(\by) q_i(\by)^2.
\end{equation}
Hence by (\ref{eq:tHpNew}),
\begin{equation}
\label{eq:thmLovasz-e2}
y_1y_2\ldots y_k p(y_1,\ldots,y_k) = \sum_{i=1}^m p_i(\by) q_i(\by)^2.
\end{equation}
Now consider arbitrary real numbers $x_1,\ldots,x_k$ not all of them zero. Set
\begin{equation}
\label{eq:choiceOfY}
\by := \left(\frac{x_1^2}{x_1^2+\ldots+x_k^2},\ldots,\frac{x_k^2}{x_1^2+\ldots+x_k^2}\right).
\end{equation}
As $p$ is homogenous, we have
$$
y_1y_2\ldots y_k p(\by) = (x_1^2+\ldots+x_k^2)^{-\deg(p)-k} x_1^2\ldots x_k^2 p(x_1^2,x_2^2,\ldots, x_k^2).
$$
Substituting this in (\ref{eq:thmLovasz-e2}) and multiplying both sides by a large enough power of $(x_1^2+\ldots+x_k^2)$ shows that
$$(x_1^2+x_2^2+\ldots+ x_k^2)^N \left( x^2_1x^2_2 \ldots x_k^2 p(x_1^2,x_2^2,\ldots, x_k^2) \right) \in \Sigma_k,$$ for some positive integer $N$.
But this contradicts the assumption that $(1,0,\ldots,0)$ is a bad point for $x^2_1x^2_2 \ldots x_k^2p(x_1^2,\ldots,x_k^2)$.
\end{proof}

\begin{proof}[Proof of Theorem~\ref{thm:Razborov}]
The proof parallels the proof of Theorem~\ref{thm:Lovasz}. Let $H,p$ and $x$ be as in the proof of Theorem~\ref{thm:Lovasz}. Then $x$ is positive. If $x$ is CS-positive, Lemma~\ref{lem:key} shows that $t(x;H,\by) = \sum_{i=1}^m p_i(\by) q_i(\by)^2$ for polynomials $p_1,\ldots,p_m \in \mathbb{R}[y_1,\ldots,y_k]$ with non-negative coefficients, and $q_1,\ldots,q_m \in \mathbb{R}[y_1,\ldots,y_k]$.  The rest of the proof proceeds as in that of Theorem~\ref{thm:Lovasz}.
\end{proof}

\section{Undecidability}
As we discussed in Section~\ref{sec:prel}  determining the positivity of a polynomial with rational coefficients is decidable.  However it follows from Matiyasevich's solution to Hilbert's tenth problem~\cite{MR0258744} that if one restricts to integer-valued variables, this problem becomes undecidable. More precisely given a multivariate polynomial with integer coefficients, the problem of determining whether it is non-negative for every assignment of integers (equivalently, positive integers) to its variables is undecidable. To prove the undecidability in Theorem~\ref{thm:undecide} we will need the following simple consequence of this fact.

\begin{lem}
\label{lem:Hilbert10}
The following problem is undecidable.
\begin{itemize}
 \item {\sc instance:} A positive integer $k \geq 6$, and a polynomial $p(x_1,\ldots,x_k)$ with integer coefficients.
 \item {\sc question:} Do there exist $x_1,\ldots,x_k \in \{1-1/n : n \in \mathbb{N}\}$ such that $p(x_1,\ldots,x_k)<0$?
\end{itemize}
\end{lem}
\begin{proof}
Consider a polynomial $q(y_1,\ldots,y_k)$ with integer coefficients. Note that $q(y_1,\ldots,y_k) \ge 0$ for all $y_1,\ldots,y_k \in \mathbb{N}$ if and only if the polynomial with integer coefficients $$p(x_1,\ldots,x_k) := \left(\prod_{i=1}^k(1-x_i)^{\deg(q)}\right) q\left(\frac{1}{1-x_1},\ldots,\frac{1}{1-x_k}\right)$$ is non-negative for all $x_1,\ldots,x_k \in \{1-1/n : n \in \mathbb{N}\}$. Hence the problem is undecidable.
\end{proof}

It is not a priori clear how the non-negativity of formulas of the form $a_1 t(H_1; \cdot)+\ldots+a_k t(H_k;\cdot)$ is related to the non-negativity of a polynomial on \emph{integers}. Indeed if one considers a single graph $H$, then the set of all possible values of $t(H;\cdot)$ is everywhere dense in the interval $[0,1]$.
The key ingredient in the proof of Theorem~\ref{thm:undecide} is the observation that there are relations between $t(H;\cdot)$ for different graphs $H$ which are satisfied  for an infinite, but sparse set of possible values of the corresponding homomorphism densities. This can be already seen in the case of the relation between the edge homomorphism density and the triangle homomorphism density. Let $g(x) := 2x^2-x$. As it is mentioned in Section~\ref{sec:intro}, Goodman~\cite{MR0107610} proved that $t(K_3;G) \ge g(t(K_2;G))$, for every graph $G$. Bollob\'as improved Goodman's bound to the following.

\begin{thm}[Bollob\'as~\cite{MR0396327}]
\label{thm:bollobas}
For every graph $G$, and every positive integer $t$, if $$t(K_2; G) \in \left[1-\frac{1}{t},1-\frac{1}{t+1}\right],$$ then we have $t(K_3; G) \ge L(t(K_2; G))$, where
\begin{equation}\label{eq:defineL}
L(x) := \frac{3t^2-t-2}{t(t+1)}x - \frac{2(t-1)}{t+1}.
\end{equation}
\end{thm}
Note that for every positive integer $t$, on the interval $\left[1-\frac{1}{t},1-\frac{1}{t+1}\right]$, $L$ is the linear function  that coincides with $g$ on the endpoints. Razborov~\cite{MR2433944} has recently proven the exact lower bound for $t(K_3; G)$ in terms of $t(K_2; G)$, but Bollob\'as's result suffices for our purpose.  Let $L: [0,1) \to \mathbb{R}$ be the continuous piecewise linear function defined on $\left[1-\frac{1}{t},1-\frac{1}{t+1}\right]$ by (\ref{eq:defineL}) for every positive integer $t$. By Theorem~\ref{thm:bollobas}, for every graph $G$, $(t(K_2;G),t(K_3;G)) \in R$ where $R  \subset [0,1]^2$ is the region defined as $$R:=\{(x,y) \in [0,1]^2: y \ge L(x)\}.$$

The examples of complete graphs show that Goodman's bound is tight when $$t(K_2; \cdot) \in \left\{1-\frac{1}{n}: n \in \mathbb{N}\right\},$$ and on the other hand Theorem~\ref{thm:bollobas} shows that it is \emph{not} tight on the rest of the interval $[0,1)$. Hence  the algebraic expression
$t(K_3;G)-g(t(K_2;G))$ can be equal to $0$ if and only if  $t(K_2; \cdot) \in \left\{1-\frac{1}{n}: n \in \mathbb{N}\right\}$. This already reveals the connection to Lemma~\ref{lem:Hilbert10} and suggests a direction for proving Theorem~\ref{thm:undecide}.

\begin{lem}
\label{lem:calculus}
Let $p$ be a polynomial in variables $x_1,\ldots,x_k$. Let $M$ be the sum of the absolute values of the coefficients of $p$ multiplied by $100\deg(p)$. Define $q \in \RR[x_1,\ldots,x_k,y_1,\ldots,y_k]$ as
$$q := p \prod_{i=1}^k (1-x_i)^6  + M \left(\sum_{i=1}^k y_i - g(x_i)\right).$$
Then the following are equivalent
\begin{itemize}
 \item {\rm (a):} $q(x_1,\ldots,x_k,y_1,\ldots,y_k) < 0$  for some $x_1,\ldots,x_k,y_1,\ldots,y_k$ with $(x_i,y_i) \in R$ for every $1 \leq i \leq k$;
 \item {\rm (b):}  $p(x_1,\ldots,x_k)<0$ for some $x_1,\ldots,x_k \in \{1-1/n : n \in \mathbb{N}\}$.
\end{itemize}
\end{lem}
\begin{proof}
If (b) holds, then for each $x_i$ we have $(x_i,g(x_i)) \in R$ and setting $y_i:=g(x_i)$ gives\newline $q(x_1,\ldots,x_k,y_1,\ldots,y_k) < 0$. Therefore (b) implies (a).

Suppose now that (a) holds. Decreasing $y_i$ decreases the value of $q$ and thus  we assume without loss of generality that $y_i = L(x_i)$. Let $$\tilde{q}(x_1,x_2,\ldots,x_k)= q(x_1,\ldots,x_k,y_1,\ldots,y_k) = p(x_1,\ldots,x_k) \prod_{i=1}^k (1-x_i)^6  + M \left(\sum_{i=1}^k L(x_i) - g(x_i)\right)<0.$$ For every $1\leq i \leq k$, let  $t_i$ be a positive integer such that $x_i \in \left[1-\frac{1}{t_i},1-\frac{1}{t_i+1}\right]$. Fixing $t_1,\ldots,t_k$ we assume that $x_1,\ldots,x_k$ are chosen in the corresponding intervals to minimize $\tilde{q}$. We claim that in this case $x_i \in  \{1-1/n : n \in \mathbb{N}\}$ for every $1\leq i \leq k$.  Suppose for a contradiction that $x_i \in \left(1-\frac{1}{t_i},1-\frac{1}{t_i+1}\right)$, for some $i$. By the choice of $x_i$ we have $\frac{\partial{\tilde{q}}}{\partial{x_i}}=0$. Hence since
$$\left| \frac{\partial}{\partial{x_i}}\left(p(x_1,\ldots,x_k)  \prod_{i=1}^k (1-x_i)^6\right)\right| \leq 7 \frac{M}{100}(1-x_i)^5 \leq \frac{M}{12t_i^5},$$
we must have
\begin{equation}\label{eq:calc1}
\frac{1}{12t_i^5} \geq |L'(x_i)-g'(x_i)| = \left| \frac{3t_i^2-t_i-2}{t_i(t_i+1)} -4x_i + 1 \right|=4 \left| \frac{t_i^2-1/2}{t_i(t_i+1)} -x_i\right|.
\end{equation}
Let $z = \frac{t_i^2-1/2}{t_i(t_i+1)}$. We can rewrite (\ref{eq:calc1}) as
$|z - x_i | \leq \frac{1}{48t_i^5}$.
Note that $L'(z) = g'(z)$,  $L'(x)-g'(x)$ is monotone on the interval between $x_i$ and $z$, and
$L(z)-g(z) = 1/(2t_i^2(t_i+1)^2)$.
It follows that
\begin{align*}
L(x_i)-g(x_i) &\geq   (L(z)-g(z)) -  |L'(x_i)-g'(x_i)||z - x_i | \\&\geq \frac{1}{2t_i^2(t_i+1)^2} - \frac{1}{48t_i^{10}} \geq \frac{1}{8t_i^4}-\frac{1}{48t_i^4} \geq \frac{1}{10t_i^4}.
\end{align*}
Finally we have,
$$\tilde{q}(x_1,x_2,\ldots,x_k) \geq -\frac{M}{100}(1-x_i)^6 + M(L(x_i)-g(x_i)) \geq M \left( \frac{1}{10t_i^4} -\frac{1}{100 t_i^6} \right) \geq 0,$$
which is a contradiction. Therefore the claim that $x_i \in  \{1-1/n : n \in \mathbb{N}\}$ for every $1\leq i \leq k$ holds. In this case we have $$0 > \tilde{q}(x_1,x_2,\ldots,x_k) = p(x_1,x_2,\ldots,x_k) \prod_{i=1}^k (1-x_i)^6,$$
which shows that (b) holds.
\end{proof}

Define the map $$\tau:\RR[x_1,\ldots,x_k,y_1,\ldots,y_k] \to \RR[v_1,\ldots,v_k,e_1,\ldots,e_k,t_1,\ldots,t_k]$$  in the following way. For every polynomial $q$,  let $\tau(q)$ be obtained from $q$ by substituting for every $i$, $e_i/v_i^2$ and $t_i/v_i^3$ instead of $x_i$ and $y_i$, respectively, and multiplying the resulting rational function by $\prod_{i=1}^k v_i^{3 \deg(q)}$ so that it becomes a polynomial.

Given a graph $H$ with $V(H)=[k]$, we define the labeled quantum graphs $V_i,E_i,T_i \in \RR[{\mathcal F}_{[k]}]$ in the following way.  For every positive integer $m$ and every $j \in [k]$, let $H_{j,m} \in {\mathcal F}_{[k]}$ be the graph on $k+m$ vertices obtained from $H$ by adding a clique of size $m$ and connecting each one of the vertices of this clique to the neighbors of the vertex $j$. Then $V_j := \sum \ind(H_{j,1} \cup F)$, $E_j := \sum_F \ind(H_{j,2} \cup F)$, and $T_j := \sum_F \ind(H_{j,3} \cup F)$, where all these sums are over different ways of joining the unlabeled vertices to the vertex $j$.

Let  $\psi_H:\RR[v_1,\ldots,v_k,e_1,\ldots,e_k,t_1,\ldots,t_k] \rightarrow \RR[\mathcal{F}_{[k]}]$ be the unique algebra homomorphism that satisfies $\psi_H(v_i) = V_i$, $\psi_H(e_i) = E_i$, and $\psi_H(t_i) = T_i$ for every $i \in [k]$.

Consider a graph $G$, and let $S$ denote the set of all maps $h:V(H) \to V(G)$ that preserve both adjacency and non-adjacency. Consider a map $\phi:V(H) \rightarrow V(G)$. If $\phi \not\in S$, then $t(\ind(H);G, \phi)=0$ which in particular shows that $t(V_i;G, \phi)=t(E_i;G, \phi)=t(T_i;G, \phi)=0$ for every $ i \in [k]$. Hence in this case  $t(\psi_H(\tau(q));G, \phi)=0$, for every polynomial $q \in \RR[x_1,\ldots,x_k,y_1,\ldots,y_k]$.

Next consider a map $\phi \in S$.  For every $j \in [k]$, let $U_j(\phi)$ be the subgraph of $G$ induced on the set of vertices $v$ for which  the map obtained from $\phi$ by replacing $\phi(j)$ by $v$ belongs to $S$. Note that for every $j \in [k]$,
$$
\mbox{$t(V_j;G, \phi) = \frac{|U_j|}{|V(G)|}$, \qquad $t(E_j;G, \phi) = t(K_2;U_j) \frac{|U_j|^2}{|V(G)|^2}$,\qquad  $t(T_j;G, \phi) = t(K_3;U_j) \frac{|U_j|^3}{|V(G)|^3}$}.$$
This shows that
$$ \mbox{$\frac{t(E_j;G, \phi)}{t(V_j;G, \phi)^2} = t(K_2;U_j)$ \qquad and \qquad $\frac{t(T_j;G, \phi)}{t(V_j;G, \phi)^3} = t(K_3;U_j)$}. $$
Recalling the definitions of $\tau$ and $\psi_H$, we conclude that
\begin{equation}
\label{eq:psiH}
t(\psi_H(\tau(q));G, \phi) =  q(t(K_2;U_1),\ldots,t(K_2;U_k),t(K_3;U_1),\ldots,t(K_3;U_k)) \prod_{j=1}^k\left(\frac{|U_j|}{|V(G)|}\right)^{3 \deg(q)},
\end{equation}
for every polynomial $q \in \RR[x_1,\ldots,x_k,y_1,\ldots,y_k]$.
\begin{claim}
\label{claim:claim1}
Let $q \in \RR[x_1,\ldots,x_k,y_1,\ldots,y_k]$ be such that $q(x_1,\ldots,x_k,y_1,\ldots,y_k) \geq 0$ for all $x_1,\ldots,x_k,y_1,\ldots,y_k$ with $(x_i,y_i) \in R$ for every $1 \leq i \leq k$. Then $\psi_H(\tau(q))$ is a positive labeled quantum graph for every graph $H$ with $V(H)=[k]$.
\end{claim}
\begin{proof}
Theorem~\ref{thm:bollobas} implies that $\left(t(K_2;U_j), t(K_3;U_j)\right) \in R$, for every $j \in [k]$. Now the claim follows from (\ref{eq:psiH}).
\end{proof}

\begin{claim}
\label{claim:claim2}
Let $q \in \RR[x_1,\ldots,x_k,y_1,\ldots,y_k]$ be such that $q(x_1,\ldots,x_k,y_1,\ldots,y_k) < 0$ for some $x_i \in \{1-1/n : n \in \mathbb{N}\}$ and $y_i = g(x_i)$ for $1\leq i \leq k$. Let $H$ be a stringent graph with $V(H)=[k]$. Then there exists a graph $G$ such that $$t(\llbracket \psi_H(\tau(q))\rrbracket; G) <0.$$
\end{claim}
\begin{proof}
Let $n_1,\ldots,n_k \in \mathbb{N}$ be so that $q$ becomes negative by setting $x_i:=1-1/n_i$ and $y_i:=2x_i^2-x_i$ for all $ i \in [k]$.  Define $G$ to be the graph obtained from $H$ by replacing the vertex $j$ of $H$ by a clique of size $n_j$ for every $j \in [k]$. Let $W_j$ be the set of the vertices of the clique that replaces the vertex $j$ of $H$ in $G$.

Let $S$ denote the set of all maps $h:V(H) \to V(G)$ that preserve both adjacency and non-adjacency. Consider a map  $\phi \in S$. It follows from the structure of $G$ that for every $j$, $\{i: \phi(i) \in W_j\}$ is a homogeneous set in $H$, and since $H$ is stringent it is of size at most $1$. (Trivially, it cannot be all of $V(H)$). Hence for every $j$, $\{i: \phi(i) \in W_j\}$ is of size exactly $1$. Since $H$ is stringent, the identity map is the only isomorphism from $H$ to itself. It follows that  $\phi(j) \in W_j$, for every $j \in [k]$. Then  for every $j \in [k]$, $U_j(\phi)$ is the restriction of $G$ to $W_j$ which by definition of $G$ is a clique of size $n_j$. Thus
$$\mbox{$t(K_2;U_j)=1-1/n_j$ \qquad and \qquad $t(K_3;U_j)=g(1-1/n_j)$},$$
which by (\ref{eq:psiH}) shows that
$$t(\psi_H(\tau(q));G, \phi) =  q(x_1,\ldots,x_k,y_1,\ldots,y_k) \prod_{j=1}^k\left(\frac{|n_j|}{|V(G)|}\right)^{3 \deg(q)}<0.$$
Moreover if $\phi \not\in S$, then $t(\psi_H(\tau(q));G, \phi)=0$. We conclude that $t(\llbracket \psi_H(\tau(q))\rrbracket; G)<0$.
\end{proof}

\begin{proof} [Proof of Theorem~\ref{thm:undecide}]
Consider an instance of the undecidable problem stated in Lemma~\ref{lem:Hilbert10}, namely a polynomial $p$ in variables $x_1,\ldots,x_k$ with integer coefficient. Construct the polynomial $q$ in variables $x_1,\ldots,x_k,y_1,\ldots,y_k$ as in Lemma~\ref{lem:calculus}. Then Lemma~\ref{lem:calculus} shows that $p(x_1,\ldots,x_k)<0$ for some $x_1,\ldots,x_k \in \{1-1/n : n \in \mathbb{N}\}$ if and only if  $q(x_1,\ldots,x_k,y_1,\ldots,y_k) < 0$  for some $x_1,\ldots,x_k,y_1,\ldots,y_k$ with $(x_i,y_i) \in R$ for every $1 \leq i \leq k$. By Claims~\ref{claim:claim1}~and~\ref{claim:claim2} determining the latter is equivalent to determining the validity of $t(\llbracket \psi_H(\tau(q))\rrbracket; \cdot) \ge 0$ where $H$ is a stringent graph on $k$ vertices. Such graphs exist and can be explicitly constructed by Lemma~\ref{lem:stringent}.
\end{proof}

\section{Proof of Theorem~\ref{thm:Artin} \label{sec:Artin}}
Before giving the proofs of Lemma~\ref{lem:preArtin} and Theorem~\ref{thm:Artin} we need to recall some facts about graphons.
Note that if  $A_G$ is the adjacency matrix of a graph $G$, then for every graph $H$
\begin{equation}
\label{eq:t_in_matrix}
t(H;G) = \Ex \prod_{uv \in E(H)} A(x_u,x_v),
\end{equation}
where  $\{x_u  \: : \: u \in V(H)\}$ are independent uniform random variables taking values in $\{1,2,\ldots,|V(G)|\}$.
Let $\mathcal{W}_0$ denote the set of bounded symmetric measurable functions of the form $w:[0,1]^2 \rightarrow [0,1]$.
The elements of $\mathcal{W}_0$ are called \emph{graphons}. For every graph $H$, and every graphon $w \in \mathcal{W}_0$, define by analogy with~(\ref{eq:t_in_matrix}),
\begin{equation}
\label{eq:t_for_w}
t(H; w) := \int \prod_{uv \in E(H)} w(x_u,x_v) \prod_{v \in V(H)} dx_v.
\end{equation}
This definition can be extended linearly to define $t(x; w)$, for every quantum graph $x$. For every graph $G$, we define a graphon $w_G \in \mathcal{W}_0$ as follows: Without loss of generality assume that $V(G)=[n]$. Then $w_G(x, y) := A_G(\lceil x n \rceil; \lceil y n \rceil)$ if $x, y \in (0; 1]$, and if $x = 0$ or $y = 0$, then $w_G(x,y):=0$. By
(\ref{eq:t_in_matrix}) and (\ref{eq:t_for_w}), for every quantum graph $x$ and graph $G$, we have
$t(x;G) = t(x;w_G)$.

A graph sequence $\{G_i\}_{i \in \mathbb{N}}$  is called convergent, if for every graph $H$, the limit $\lim_{i \rightarrow \infty} t(H;G_i)$ exists.
It is shown in~\cite{MR2274085} that for every convergent graph sequence $\{G_i\}_{i \in \mathbb{N}}$, there exists a graphon $w$ such that $\lim_{i \rightarrow \infty} t(H;G_i) = t(H;w)$, for every graph $H$.  On the other hand for every graphon $w$, it is easy to construct a convergent graph sequence $\{G_i\}_{i \in \mathbb{N}}$ such that $t(H;w)=\lim_{i \rightarrow \infty} t(H;G_i)$, for every graph $H$.

\begin{proof}[Proof of Lemma~\ref{lem:preArtin}]
If $x$ is not positive, then there exists a graphon $w$ such that $t(x; w) <0$. Since $g \neq 0$,  by (\ref{eq:whitney}) there exists a graphon $w'$ satisfying $t(g;w') \neq 0$. Now note that by (\ref{eq:t_for_w}), $t(g; \alpha w' + (1-\alpha) w)$  is a polynomial in $\alpha$. This polynomial is not identically $0$ as it is not equal to zero on $\alpha=1$. Hence there are arbitrarily small $\alpha>0$ for which  $t(g; \alpha w' + (1-\alpha) w) \neq 0$. By taking a sufficiently small such $\alpha$, we obtain a graphon $w'':= \alpha w' + (1-\alpha) w$ that satisfies both
$$\mbox{$t(g;w'')>0$ \qquad and \qquad $t(x;w'')<0$}.$$
This in particular implies that $t(gx; w'') <0$, contradicting $gx=h$ and the assumption that $h$ is expressible as a sum of squares.
\end{proof}

\begin{proof}[Proof of Theorem~\ref{thm:Artin}]
Suppose to the contrary that for every positive  $x \in \A$, there exist $g, h \in \A$, $g \neq 0$, each expressible as sums of squares, such that $gx=h$. We will show  this would imply that given a quantum graph $f$ with rational coefficients, the problem of determining the validity of $f \ge 0$ is decidable, contradicting Theorem~\ref{thm:undecide}.

Given a collection of partially labeled graphs $\G = (G_1,\ldots,G_k)$, an integer $m$ and a matrix $A=(a_{ij} \: : \: i \in [m], j \in [k] )$ define $$z(\G,m,A):=\left\llbracket \sum_{i=1}^{m} \left( \sum_{j=1}^k a_{ij}G_j \right)^2\right\rrbracket.$$
This quantum graph can be expressed as a linear combination of graphs of the form $\llbracket G_{j_1} \cdot G_{j_2}\rrbracket$, where $j_1,j_2 \in [k]$, with coefficients polynomial in the entries of $A$. Note that connected non-isomorphic graphs  are algebraically independent as elements of $\A_{\emptyset}$ and every graph as an element of $\A_{\emptyset}$ is equal to the product of its connected components. It follows that for a fixed quantum graph $f$ with rational coefficients, a fixed collection of partially labeled graphs $\G$ and an integer $m$,  the system
\begin{equation}\label{eq:sosArtin}
z(\G,m,A)f=z(\G,m,B) \qquad \mathrm{and} \qquad z(\G,m,A)\neq 0
\end{equation}
can be expressed as a (computable) system of polynomial equations and inequalities with rational coefficients on the entries of $A$ and $B$. Therefore, it is possible to decide whether there exist matrices $A$ and $B$ with real entries solving this system.


Hence in order to decide the validity of $f \ge 0$, one enumerates  finite graphs $G$ and checks the validity of $t(f;G) \ge 0$ on each graph. In parallel, one enumerates all pairs $(\G,m)$, where $\G$ is a finite sequence of finite partially labeled graphs and $m$ is an integer, and for each such pair checks whether there exists a solution to (\ref{eq:sosArtin}).
\end{proof}

\section*{Acknowledgements}
The authors wish to thank Alexander Razborov for many enlightening discussions. They also wish to thank  Swastik Kopparty, L{\'a}szl{\'o} Lov{\'a}sz and anonymous referee  for their valuable comments.
%

\bibliographystyle{alpha}
\bibliography{positive}
\appendix
\section{Original formulation of Razborov's Cauchy-Schwarz calculus. \label{appendix}}

In the formulation given in~\cite{MR2371204}, the Cauchy-Schwarz calculus contains an additional axiom and an additional inference rule.
The axiom can be stated here as:
\begin{itemize}
\item $\ind(H) \geq 0$ for every partially labeled graph $H$.
\end{itemize}
This axiom can be derived in the Cauchy-Schwarz calculus  presented in  Definition~\ref{def:CS} as follows.
Let $H'$ be a graph obtained from $H$ by assigning new labels to the previously unlabeled vertices of $H$, so that all vertices of $H'$ are labeled. Then we have $(\ind(H'))^2=\ind(H')$. It follows that $\ind(H) = \llbracket (\ind(H'))^2 \rrbracket_T$, where $T$ is the set of the labels used on the vertices of $H$. Thus the above axiom follows from the axiom A1 and the inference rule R3 in Definition~\ref{def:CS}.

Let $f$ be a labeled quantum graph, and let $H$ be a graph, with all vertices of $H$ labeled. We say that $f$ is \emph{$H$-rooted}, if there exists a positive integer $k$, partially labeled graphs $G_1,G_2,\ldots,G_k$ and real numbers $\alpha_1,\ldots,\alpha_k$ such that $f = \sum_{i=1}^k \alpha_i \ind(G_i)$ and in every $G_i$ the subgraph induced by the labeled vertices is equal to $H$. An additional inference rule from~\cite{MR2371204} can now be stated as follows:
\begin{itemize}
\item Let $f$ be an $H$-rooted labeled quantum graph with $f \geq 0$. Let $H'$ be a graph with all vertices labeled, such that $H$ is an induced subgraph of $H'$ as a labeled graph.  Then $f\cdot \ind(H') \geq 0$.
\end{itemize}
Note that this inference rule is subsumed in inference rule R3 of Definition~\ref{def:CS}.

Let us further note that in~\cite{MR2371204} the product of two labeled quantum graphs is  only defined if both of the graphs are $H$-rooted for some $H$. In our framework two partially labeled graphs correspond to the same element of $\mathcal{A}$ if one is obtained from another by adding isolated, possibly labeled, vertices. Therefore we can consider every labeled quantum graph as a linear combination of partially labeled graphs, all of which have exactly the same set of labeled vertices. Thus every labeled quantum graph $f$ can be written in a form $f = \sum_{H} f_H$, where the summation is taken over all labeled graphs $H$ with $V(H)=[l]$ for some positive integer $l$, so that the vertices of $H$ are labeled in the natural way, and each $f_H$ is $H$-rooted. One can routinely deduce from the definitions that $f \geq 0$ if and only if $f_H \geq 0$ for every $H$. Further, it follows from (\ref{eq:orthogonality}) that if $f = \sum_{H} f_H$ and $g=\sum_{H}g_H$ are as above,  then $fg=\sum_{H}f_Hg_H$. Consequently, our multiplication inference rule could be restricted as in~\cite{MR2371204} to multiplying only $H$-rooted quantum graphs.

\end{document}